\documentclass[12pt]{article}

\usepackage{amsmath,amssymb,amsthm,amsfonts,mathtools,dsfont}
\usepackage{enumerate}
\usepackage[english]{babel}
\usepackage{fullpage}
\usepackage{xcolor}

\numberwithin{equation}{section}

\newtheorem{theorem}{Theorem}[section]
\newtheorem{lemma}[theorem]{Lemma}
\newtheorem{corollary}[theorem]{Corollary}

\theoremstyle{remark}
\newtheorem{definition}[theorem]{Definition}
\newtheorem{remark}[theorem]{Remark}

\DeclareMathOperator{\dif}{\mathrm{d}\!}

\renewcommand{\k}{k}
\newcommand{\m}{m}
\newcommand{\n}{n}
\renewcommand{\u}{u}
\renewcommand{\v}{v}
\newcommand{\x}{x}
\newcommand{\y}{y}
\newcommand{\z}{z}
\newcommand{\1}{\bar 1}
\newcommand{\E}{{\mathbf{E}}}
\renewcommand{\P}{\mathbf{P}}
\newcommand{\Prob}[1]{\mathbf{P}\left\{#1\right\}}
\newcommand{\R}{{\mathbb{R}}}
\newcommand{\Sphere}[1][d-1]{{\mathbb{S}}^{#1}}
\newcommand{\NN}{{\mathbb{N}}}
\newcommand{\Rp}[1][d]{\R_{\scriptscriptstyle ++}^{#1}}
\newcommand{\sM}{{\mathcal{M}}}
\newcommand{\sD}{{\mathcal{D}}}
\newcommand{\sH}{{\mathcal{H}}}
\newcommand{\sO}{{\mathcal{O}}}
\newcommand{\ssO}{{\scriptstyle\sO}}
\newcommand{\eps}{\varepsilon}
\renewcommand{\kappa}{\varkappa}
\newcommand{\cone}{\mathbb{T}}
\newcommand{\econe}{\mathbb{\widehat T}}
\newcommand{\kcone}{{\mathbb T_K}}
\newcommand{\ekcone}{\cone}
\newcommand{\cecone}{\econe^c}

\newcommand{\rhoH}{\rho_{\mathrm{H}}}

\newcommand{\fnc}{\kappa}

\DeclareMathOperator{\card}{card}

\newlength{\querylen}
\usepackage{fancybox}

\begin{document}

\title{Limit theorems for multidimensional renewal sets}
	
\author{Andrii Ilienko and Ilya Molchanov\\
{\small\it Igor Sikorsky National Technical University of Ukraine (KPI), 
  ilienko@matan.kpi.ua}\\
{\small\it University of Bern, ilya@stat.unibe.ch}}

\date{\today}

\maketitle

\begin{abstract}
  Consider multiple sums $S_n$ of i.i.d.~random variables with a
  positive expectation on the $d$-dimensional integer grid. We prove
  the strong law of large numbers, the law of the iterated logarithm
  and the distributional limit theorem for random sets $\sM_t$ that
  appear as inversion of the multiple sums, that is, as the set of all
  arguments $x\in\R_+^d$ such that the interpolated multiple sum $S_x$
  exceeds $t$. The moment conditions are identical to those imposed in
  the limit theorems for multiple sums. The results are
  expressed in terms of set inclusions and using distances between
  sets.
\end{abstract}

\section{Introduction}

Classical renewal theorems can be viewed as inverse results to
limit theorems for sums of i.i.d.~random
variables.  In this paper we consider similar results for multiple
sums $S_n$ on the $d$-dimensional grid $\NN^d$.  Unless otherwise
noted, assume $d\ge2$.

The letters $\m$, $\n$, $\k$, and
$\u,x,y,\z$ stand for vectors from $\NN^d$ or $\R_+^d=[0,\infty)^d$,
or of spaces of other dimensions. Their components are
denoted by the respective superscripted letters,
e.g., $\m=(m^1,\dots,m^d)$. Denote $\1=(1,\dots,1)$.

We will also make use of the standard componentwise partial order
with $\m\le\n$ meaning that $m^i\le n^i$ for all $i=1,\dots,d$, 
denote
\begin{displaymath}
  |\m|=m^1\cdots m^d 
\end{displaymath}
and write $\m\to\infty$ if $\max\{m^1,\dots,m^d\}\to\infty$. For
$m\in\NN^d$ this is the case if and only if $|\m|\to\infty$, while the
condition $\y\to\infty$ does not imply $|\y|\to\infty$ for
non-integer $\y\in\R_+^d$.

Let $\{\xi_\m$, $\m\in\NN^d\}$, be a multi-indexed family of independent
copies of an integrable random variable $\xi$ with finite mean
$\mu=\E\xi>0$. Denote by 
\begin{displaymath}
  S_\n=\sum_{\m\le\n}\xi_\m,\quad \n\in\NN^d,
\end{displaymath}
the corresponding multiple sums, and let $S_\n=0$
for $n$ with at least one vanishing component. It is convenient to
extend these multiple sums to all indices $\x\in\R_+^d$ by the \emph{piecewise
multi-linear interpolation}, see, e.g., \cite{Weiser}. Let
\begin{equation}
  \label{int}
  S_\x=\sum_{\k\in C_\x}v_\k(\x)\,S_{\k^\ast}, \quad x\in\R_+^d,
\end{equation}
where $C_\x$ denotes the set of all vertices of the unit cube 
which contains $x$, $v_\k(\x)$ is the volume of the box with $\k$ and
$\x$ being diagonally opposite vertices and with faces parallel to the
coordinate planes, and $\k^\ast$ means the vertex opposite to $\k$ in
the cube that contains $x$. 
It is easily seen that (\ref{int}) determines
$S_\x$ uniquely even if $\x$ lies on the boundaries of several
adjacent cubes. This interpolation technique, expressed in another
way, was used by Bass and Pyke \cite{bas:pyk84f}. A special feature of
this choice of interpolation is that $|x|=x^1\cdots x^d$ (being a
multilinear function in all coordinates) admits the exact
interpolation.

Consider the \emph{renewal sets} 
\begin{displaymath}
  \sM_t=\{\x\in\R_+^d:\; S_\x\ge t\},\quad t>0.
\end{displaymath}
Since the multi-linear interpolation \eqref{int} produces a continuous
function, $\sM_t$ is a random closed set in $\R^d$, see \cite{mo1}.

The strong laws of large numbers (SLLNs) for multiple sums were established in \cite{smy73} and
\cite{Gut}. Unlike the conventional case of $d=1$, they hold if and only
if the generic summand has a logarithmic moment whose order depends on
the dimension, see \eqref{SLLN}. By inverting this and other SLLNs for
multiple sums, we show that the rescaled random sets $t^{-1/d}\sM_t$ converge
as $t\to\infty$ (in a sense to be specified) to the set
\begin{equation}
  \label{H}
  \sH=\left\{\x\in\R_+^d:\; |\x|\ge\mu^{-1}\right\}.
\end{equation}

The law of the iterated logarithm (LIL) for renewal sets deals with
modifications of $\sH$ obtained by perturbing $\mu^{-1}$ with an
iterated logarithm term multiplied by a constant. We examine the
values of the constant that ensure the validity of the LIL and show
that the boundary values violate it. We also derive the LILs for
distances between the scaled $\sM_t$ and $\sH$. While the upper limits
are non-trivial, it is shown that the lower limits vanish. The latter
is rather suprising meaning that, inside any cone, the boundary of $t^{-1/d}\sM_t$
infinitely often lies within a small envelope around the boundary of
$\sH$. The proof relies on considering the LILs for multiple sums
inside a cone, outside it and subtle results concerning the LIL for
subsequences. Finally, we establish the central limit theorem for
radial functions that represent $\sM_t$ in the spherical coordinates.

The longer proofs of the SLLN and the LIL are postponed to separate
sections. Special features of the one-dimensional case are considered
in Section~\ref{sec:one-dimensional-case}.  In Appendix, we derive a
strong law of large numbers and a law of the iterated logarithm for
multi-dimensional sums $S_\n$ as $\n\to\infty$ within a sector. These
results differ from those available in the literature so far and
complement the sectorial laws proved in \cite{GutSect}.

Similar results hold for sums generated by marked Poisson point
processes, where $S_x$ is the sum of the marks for the points
dominated by $x\in\R_+^d$. 

Throughout the paper, $\log c$ and $\log\log c$ for $c\ge0$ have the
usual meanings except near zero; we set $\log c$, resp.~$\log\log c$,
to be $1$ over $[0,e)$, resp.~$[0,e^e)$. The extended logarithmic
functions become positive and monotone on $\R_+$.

\section{Strong law of large numbers}
\label{sec:strong-law-large}

We start with a rather general multidimensional inversion theorem
which allows converting a.s.\ limit theorems for $S_\n$ to their
counterparts for $\sM_t$ in terms of set inclusions. We will need the
following generalisation of the regular variation property, which is
due to Avacumovi\'c \cite{Avac}, see also \cite{Alj,BuldBook}
and references therein.

\begin{definition}
  A function $p:[0,\infty)\mapsto[0,\infty)$ which is positive for all
  sufficiently large arguments is said to be $\sO$-regularly
  varying if, for all $c>0$,
  \begin{displaymath}
    \limsup_{t\to\infty}\frac{p(ct)}{p(t)}<\infty.
  \end{displaymath}
\end{definition}

The class of $\sO$-regularly varying functions includes all regularly
varying functions and many oscillating ones.  The substitution $c\to
c^{-1}$ leads to an equivalent characterisation:
\begin{equation}
  \label{ORV_inf}
  \liminf_{t\to\infty}\frac{p(ct)}{p(t)}>0.
\end{equation}

For $c\in\R$, denote
\begin{displaymath}
  \sH(c)=\left\{\x\in\R_+^d:\; |\x|\ge\mu^{-1}+c\right\}.
\end{displaymath}
Then $\sH(c)$ decreases in $c$, and $\sH(0)$ becomes $\sH$ from \eqref{H}. 

\begin{theorem}[Multidimensional inversion]
  \label{inv_th}
  Let $p$ be an $\sO$-regularly varying function such that $p(t)$ is
  non-decreasing and $t^{-1}p(t)$ is non-increasing for all sufficiently
  large $t$.  If
  \begin{equation}
    \label{hyp}
    S_\n-\mu|\n|=\ssO(p(|\n|))\quad \text{a.s.~as }\n\to\infty,
  \end{equation}
  then, for all $\eps>0$ and sufficiently large $t$, 
  \begin{equation}
    \label{incl}
    \sH(\eps p(t)t^{-1})\subset t^{-1/d}\sM_t
    \subset \sH(-\eps p(t)t^{-1})\quad\text{a.s.}
  \end{equation}
\end{theorem}

Theorem~\ref{inv_th} yields the following Marcinkiewicz--Zygmund type
SLLN for $\sM_t$ in terms of set inclusions.

\begin{corollary}[SLLN for renewal sets, set-inclusion version]
  \label{set_SLLN_cor}
  If
  \begin{equation}
    \label{SLLN}
    \E(|\xi|^{\beta}\log^{d-1}|\xi|)<\infty
  \end{equation}
  for some $\beta\in[1,2)$, then, for each $\eps>0$ and all
  sufficiently large $t$,  
  \begin{displaymath}
    \sH(\eps t^{-1+1/\beta})\subset t^{-1/d}\sM_t
    \subset \sH(-\eps t^{-1+1/\beta})\quad\text{a.s.}
  \end{displaymath}
\end{corollary}
\begin{proof}
  According to the Marcinkiewicz--Zygmund type SLLN for multi-indexed
  sums due to Gut \cite[Th.~3.2]{Gut} (see also \cite[Cor.~9.3]{KlesBook}),
  (\ref{SLLN}) implies (\ref{hyp}) with the required function
  $p(t)=t^{1/\beta}$, $t>0$, which satisfies the conditions of
  Theorem~\ref{inv_th}. To be more precise, in Gut's paper
  $\n\to\infty$ means $\min\{n_1,\dots,n_d\}\to\infty$ instead of 
  $\max\{n_1,\dots,n_d\}\to\infty$. However, the necessary refinement
  can be easily obtained. 
\end{proof}

Theorem~\ref{inv_th} yields further strong laws of large
numbers under other normalisations that still ensure the
validity of the SLLNs for multiple sums as described in
\cite[Ch.~9]{KlesBook}.

In the following, $\cone$ denotes a closed convex cone such that
\begin{equation}
  \label{eq:cone-sub}
  \cone\setminus\{0\}\subset\Rp=(0,\infty)^d.
\end{equation}
If \eqref{hyp} is weakened to
\begin{displaymath}
S_\n-\mu|\n|=\ssO(p(|\n|))\quad \text{a.s.~as }\cone\ni\n\to\infty
\end{displaymath}
for all such cones $\cone$ (where $\cone\ni\n\to\infty$ means that
$n\to\infty$ within $\cone$), then \eqref{incl} is replaced by
\begin{equation}
  \label{incl-bis}
  \cone\cap \sH(\eps p(t)t^{-1})\subset \cone\cap t^{-1/d}\sM_t
  \subset\cone\cap \sH(-\eps p(t)t^{-1}).
\end{equation}
The proof of (\ref{incl-bis}) follows the lines of the proof of
Theorem~\ref{inv_th}, see Section~\ref{sec:proofs-results-sect}.
These conical (or sectorial) versions of the a.s.~limit theorems
usually hold under weaker moment assumptions. The next result follows
from the sectorial SLLN proved in Theorem~\ref{S_set_SLLN_lem}.

\begin{corollary}
  \label{cor:sector-SLLN}
  If $\cone\setminus\{0\}\subset\Rp$ and $\E|\xi|^\beta<\infty$ for some
  $\beta\in[1,2)$, then \eqref{incl-bis} holds with
  $p(t)=t^{1/\beta}$.
\end{corollary}

Note that \eqref{incl} implies that $t^{-1/d}\sM_t\to\sH$ almost
surely in the Fell topology on the family of closed sets, see, e.g.,
\cite[Appendix~C]{mo1}.  The convergence of sets can be quantified in
various ways. The \emph{Hausdorff} distance between two subsets $X$
and $Y$ of $\R^d$ is defined by
\begin{displaymath}
  \rhoH(X,Y)=\max\left\{\sup_{\x\in X}\inf_{\y\in Y}\rho(\x,\y),\,
  \sup_{\y\in Y}\inf_{\x\in X}\rho(\x,\y)\right\},
\end{displaymath}
with $\rho$ denoting the Euclidean distance in $\R^d$.

The \emph{localised symmetric difference} distance (also called the
Fr\'echet--Nikodym distance) between two Borel subsets $X$ and $Y$ of
$\R^d$ is defined by
\begin{displaymath}
  \rho_\triangle^K(X,Y)=\lambda_d(K\cap(X\triangle Y)),
\end{displaymath}
where $\lambda_d$ is the Lebesgue measure on $\R^d$ and $K$ is a Borel
set in $\R^d$ that determines the localisation. 

\begin{theorem}[SLLN for renewal sets, metric version]
  \label{met_SLLN_th}
  If \eqref{SLLN} holds for some $\beta\in[1,2)$, then
  \begin{equation}
    \rhoH(t^{-1/d}\sM_t,\sH)=\ssO(t^{-1+1/\beta})
    \quad\text{a.s.~as }t\to\infty,
    \label{H_bound}
  \end{equation}
  and, for any compact set $K\subset\R^d$,
  \begin{equation}
    \rho_\triangle^K(t^{-1/d}\sM_t,\sH)=\ssO(t^{-1+1/\beta})
    \quad\text{a.s.~as }\; t\to\infty.
    \label{D_bound}
  \end{equation}
  If, additionally, $K\subset\Rp$, then (\ref{D_bound}) holds provided
  only that $\E|\xi|^\beta<\infty$. Under this condition,
  \eqref{H_bound} holds for $\rhoH((t^{-1/d}\sM_t)\cap K,\sH\cap K)$. 
\end{theorem}

We now briefly consider discrete renewal sets $\sM_t\cap\NN^d$
constructed by non-interpolated partial sums. Strong limit
theorems for the cardinality $N_t$ of the finite set
$\NN^d\setminus\sM_t$ may be found in \cite[Ch.~11]{KlesBook}. In
particular, the following SLLN holds.

\begin{theorem}[see \protect{\cite[Th.~11.7]{KlesBook}}]
  Let $\xi\ge0$ a.s. If
  $\E(\xi\log^{d-1}\xi)<\infty$, then 
  \begin{displaymath}
    \frac{N_t}{t\log^{d-1}t}\to \frac 1{\mu(d-1)!}
    \quad \text{a.s. as }\; t\to\infty. 
  \end{displaymath}
\end{theorem}

A similar result holds for $\E N_t$, see \cite[Th.~11.5]{KlesBook}.
Set-inclusion results for $t^{-1/d}(\sM_t\cap\NN^d)$
immediately follow from those for the continuous renewal sets, e.g.,
(\ref{incl}) holds with all sides intersected with $t^{-1/d}\NN^d$.
The situation with metric results is more complicated. In the most
natural form, these results would look like a.s.~limit theorems for
the number of lattice points in $(t^{-1/d}\sM_t)\triangle\sH$. Such theorems
might be derived from discretised set inclusions \eqref{incl} by
using bounds on the number of integer points between the sets
$\partial\sH(c)$ for different $c$'s. The latter, in turn, are closely
related to the so-called generalised Dirichlet divisor problem in
number theory.

For completeness, we now give some facts on this topic, following
\cite[Appendix~10]{KlesBook}. For $k\geq1$, let
\begin{displaymath}
  T_k=\card\{n\in\NN^d:\; |n|\le k\}.
\end{displaymath}
In order to bound the number of integer points between the sets
$\partial\sH(c)$, we need some results on the asymptotic behaviour of
$T_k-T_j$ as $j,k\to\infty$.  It can be proved that there exists a
polynomial $\mathcal P_d$ of degree $d-1$ such that
\begin{equation*}
  T_k=k\mathcal P_d(\log k)+\ssO(k^\alpha)\quad \text{as }\; k\to\infty,
\end{equation*}
for all $\alpha>\alpha_d$ with some $\alpha_d<1$. Although there is a
number of results concerning $\alpha_d$, their exact values are not
yet known. According to the Hardy--Titchmarsh conjecture (that would
follow from the Riemann hypothesis), $\alpha_d=(d-1)/(2d)$, and this
bound would be sufficient in order not to dominate the stochastic
factors.
Without involving this and related number-theoretic
conjectures, the necessary bounds can be obtained only in the case
$d=2$.

\section{Laws of the iterated logarithm}
\label{sec:law-iter-logar}

Now we
turn to the law of the iterated logarithm (LIL) for $\sM_t$ in terms
of set inclusions. Recall that $\cone$ always denotes a closed convex
cone such that \eqref{eq:cone-sub} holds. Let
\begin{equation}
  \label{Hast}
  \sH_\cone(c)=(\cone\cap\sH(c))\cup
  ((\Rp\setminus\cone)\cap\sH(c\sqrt d)).
\end{equation}
In other words, $\sH_\cone(c)$ consists of all points $\x\in\Rp$ such
that $|\x|\geq\mu^{-1}+c$ in case $\x\in\cone$ and $|\x|\geq
\mu^{-1}+c\sqrt{d}$ if $\x\notin\cone$.

Assume that $\xi$ has a finite variance denoted by $\sigma^2$ and denote 
\begin{displaymath}
  \fnc(t)=\sqrt{2 t^{-1}\log\log t}, \quad t>0.
\end{displaymath}

\begin{theorem}[LIL for renewal sets, set-inclusion version]
  \label{LIL}
  Let 
  \begin{equation}
    \label{Wich_cond}
    \E\left[\xi^2\frac{\log^{d-1}|\xi|}{\log\log|\xi|}\right]<\infty.
  \end{equation}
  \begin{enumerate}[(i)]
  \item If $\gamma<-\mu^{-3/2}$, then 
    \begin{displaymath}
      t^{-1/d}\sM_t\subset\sH_\cone(\gamma\sigma\fnc(t))\quad \text{ a.s.}
    \end{displaymath}
    for all sufficiently large $t$.
  \item If $-\mu^{-3/2}\le\gamma\le\mu^{-3/2}$, then there are
    sequences $\{t'_i,i\geq1\}$ and $\{t''_i,i\geq1\}$
    depending on $\omega$, $\cone$, and $\gamma$ such that
    $t'_i\to\infty$ and $t''_i\to\infty$ almost surely, and
    \begin{align}
      \label{notincl1}
      (t'_i)^{-1/d}\sM_{t'_i}&\not\subset\sH_\cone(\gamma\sigma\fnc(t'_i)),\\
      (t''_i)^{-1/d}\sM_{t''_i}&\not\supset\sH_\cone(\gamma\sigma\fnc(t''_i)),
      \label{notincl2}
    \end{align}
    almost surely for all $i$. 
  \item If $\gamma>\mu^{-3/2}$, then
    \begin{displaymath}
      t^{-1/d}\sM_t\supset\sH_\cone(\gamma\sigma\fnc(t))\quad \text{ a.s.}
    \end{displaymath}
    for all sufficiently large $t$.
  \end{enumerate}
\end{theorem}

The idea of the proof of this theorem is to apply two laws of the
iterated logarithm for multiple sums. First, a modification
of the sectorial law from \cite{GutSect} with the limiting constant $1$
(proved in Appendix) is applicable inside $\cone$, while the law of
the iterated logarithm from \cite{Wich} in the full $\R_+^d$ with the
limiting constant $\sqrt{d}$ is applicable in the complement of
$\cone$.

\begin{remark}
  \label{LIL_rem}
  Theorem~\ref{LIL} may be reformulated as 
  \begin{align*}
    \sup\left\{\gamma:\; t^{-1/d}\sM_t\subset\sH_\cone
    (\gamma\sigma\fnc(t))
    \text{ a.s.~for large $t$}\right\}&=-\mu^{-\frac 32},\\
    \inf\left\{\gamma:\; t^{-\frac 1d}\sM_t\supset\sH_\cone
    (\gamma\sigma\fnc(t))
    \text{ a.s.~for large $t$}\right\}&=\mu^{-\frac 32},
  \end{align*}
  and the supremum and infimum are not attained in the sense that the
  above inclusions do not hold for the critical values $\pm\mu^{-3/2}$.
\end{remark}

As previously, we now quantify the results of Theorem~\ref{LIL} by
means of the Hausdorff distance $\rhoH$ and the localised symmetric
difference metric $\rho_\triangle^K$. For any cone $\cone$, define
\begin{equation}
  \label{LK}
  L_\cone=\frac{1}{d} \int_{\cone\cap\Sphere} 
  |u|^{-1}\,\dif u,
\end{equation}
where $\Sphere$ is the unit Euclidean sphere in $\R^d$. 
For compact set $K\subset\R_+^d$, let $\kcone$ denote the cone
generated by $K\cap \partial\sH$, that is, the smallest cone
containing $K\cap\partial\sH$. Note that $\kcone$ satisfies 
(\ref{eq:cone-sub}).

\begin{theorem}[LIL for renewal sets, metric version]
  \label{met_LIL_th}
  Under the assumptions of Theorem~\ref{LIL},
  \begin{equation}
    \label{met_LIL_H}
    \limsup_{t\to\infty}\frac{\rhoH(t^{-1/d}\sM_t,\sH)}{\fnc(t)}
    =d^{-\frac 12}\sigma\mu^{-\frac 12-\frac 1d}\quad\text{a.s.},
  \end{equation}
  and, for any compact set $K$ in $\R^d$, 
  \begin{equation}
    \label{met_LIL_tr}
    \limsup_{t\to\infty}\frac{\rho_\triangle^K(t^{-1/d}\sM_t,\sH)}{\fnc(t)}
    \le 2\sigma\mu^{-\frac 32}L_\kcone\quad\text{a.s.}
  \end{equation}
  If $\xi$ is a.s.~non-negative, (\ref{met_LIL_tr}) holds with the
  factor $2$ on the right-hand side replaced by $1$. 
  If, additionally, $K\subset\Rp$, then (\ref{met_LIL_tr}) 
  holds provided only that $\E\xi^2<\infty$.
\end{theorem}

Note that (\ref{met_LIL_H}) gives the exact value of the upper limit
unlike \eqref{met_LIL_tr}.  This is
achieved due to the high sensitivity of the Hausdorff metric to
outlying points. 

Assume that $\xi$ is a.s.~non-negative. In the one-dimensional case, the corresponding \emph{lower} limits in Theorem \ref{met_LIL_th} equal zero.
Indeed, it follows from the ordinary LIL and continuity of $S_x$ that $S_{t_i/\mu}=t_i$ along some sequence $t_i\to\infty$. Since $\xi\ge0$ a.s., this implies
\begin{displaymath}
t_i^{-1}\sM_{t_i}=[\mu^{-1},\infty)=\sH\quad\text{for all $i\ge1$},
\end{displaymath}
and the claim follows. It is quite remarkable that this, even in a stronger form, remains true in any dimension.

\begin{theorem}
  \label{liminf_th}
  Let $\xi$ be a.s.~non-negative. If \eqref{Wich_cond} holds, then
\begin{gather}
\label{liminf_H}
\liminf_{t\to\infty}\sqrt t\,\rhoH(t^{-1/d}\sM_t,\sH)=0\quad\text{a.s.,}\\
\label{liminf_tr}
\liminf_{t\to\infty}\sqrt t\,\rho_\triangle^K(t^{-1/d}\sM_t,\sH)=0
\quad\text{a.s.}
\end{gather}
If, additionally, $K\subset\Rp$, then (\ref{liminf_tr}) 
holds provided only that $\E\xi^2<\infty$.
\end{theorem}

\section{Convergence in distribution}
\label{sec:conv-distr}

Assume that $\sigma^2=\E(\xi-\mu)^2<\infty$.  The limit theorem for
multiple sums by Wichura \cite[Cor.~1]{wic69} yields that
\begin{equation*}
\bar S_{t,x}=\frac{S_{[tx]}-|[tx]|\mu}{\sigma t^{d/2}}, \quad x\in[0,1]^d,
\end{equation*}
converges in distribution as $t\to\infty$ to the Chentsov field $Z_x$,
$x\in[0,1]^d$, which is a centred Gaussian field with the covariance
\begin{displaymath}
\E (Z_xZ_y) =|x\wedge y|,\quad x,y\in\R_+^d. 
\end{displaymath}
Here the integer part $[\cdot]$ and the minimum $\wedge$ of vectors
are defined componentwise. The convergence of $\bar S_{t,x}$ means
that the value of each measurable functional continuous in the uniform
metric converges in distribution to its value on the limiting Chentsov
random field, see \cite[Def.~1]{wic69}.

Bickel and Wichura \cite{bic:wic71} formalised this convergence as the
weak convergence in the Skorokhod topology for random fields. The
setting in \cite{wic69} and \cite{bic:wic71} concerned the
non-interpolated fields.  The same convergence holds also for the
interpolated fields
\begin{equation}
\label{int_field}
\tilde S_{t,x}=\frac{S_{tx}-|tx|\mu}{\sigma t^{d/2}}, \quad x\in[0,1]^d.
\end{equation}
By \cite[Th.~2]{bic:wic71} or \cite[Th.~5.6]{straf72}, this
follows from the weak convergence of finite-dimensional
distributions and the tightness criterion
\begin{equation}
\label{tight}
\lim_{\delta\downarrow 0}\limsup_{t\to\infty}\Prob{w_\delta(\tilde S_{t,\cdot})>\eps}=0
\end{equation}
for any $\eps>0$. Here $w_\delta$ stands for the $\delta$-modulus of
continuity.  The finite-dimensional convergence follows from the
central limit theorem, whereas \eqref{tight} holds by the inequality
$w_\delta(\tilde S_{t,\cdot})\le w_{2\delta}(\bar S_{t,\cdot})$, which
is valid for large $t$, and the counterpart of \eqref{tight} for $\bar
S_{t,\cdot}$, which is derived in \cite[Th.~3]{wic69} and
\cite[Th.~5]{bic:wic71}.

Notice that Bass and Pyke \cite{bas:pyk84f} considered random signed
measures generated by the interpolated fields and established the
convergence in the analogue of the uniform metric for set-indexed
functions. The convergence of $\tilde S_{t,x}$ might be also directly
derived from \cite[Th.~7.1]{bas:pyk84f} under a slightly stronger moment
assumption $\E|\xi|^{2+\delta}<\infty$ for some $\delta>0$, see
Remark~8.5 ibid. We also note that the above convergence holds if
$[0,1]^d$ is replaced by any compact set $K\subset\R_+^d$. Finally, we
remark that both the pre-limiting and limiting fields are
a.s.~continuous, and so the convergence can also be regarded as the
weak convergence in the uniform metric, see \cite[p.~151]{Bil}.

The lack of a well-defined centring (and subtraction) for random sets
makes it necessary to express limit theorems for the random sets
$t^{-1/d}\sM_t$ in terms of some real-valued functions of them. For
this, choose the radial function
\begin{displaymath}
r_t(u)=\inf\{a>0\colon au\in t^{-1/d}\sM_t\}, \quad u\in\Rp.
\end{displaymath}
In this section we will assume that the generic summand $\xi$ defining the multiple sums is almost
surely non-negative. Hence, $S_{au}\leq S_{bu}$ for $a\leq b$, and so
the radial function uniquely identifies the set $\sM_t$. 

By Corollary~\ref{set_SLLN_cor}, 
\begin{displaymath}
r_t(u)\to (\mu|u|)^{-1/d}\quad \text{as }\; t\to\infty
\end{displaymath}
for all $u\in\Rp$. We may assume that the Euclidean norm of $u$ equals one.

\begin{theorem}
	\label{thr:radial}
	Assume that $\xi\ge0$ a.s.
	Let $K$ be a compact subset of $\Sphere\cap\Rp$ and let
	$f^-,f^+\colon K\mapsto\R$ be continuous functions. Then
	\begin{multline*}
	\Prob{f^-(u)<\sqrt{t}\left((r_t(u))^d-\frac{1}{\mu|u|}\right)\leq f^+(u),\;
		u\in K}\\
	\to \Prob{f^-(u)<\frac{\sigma}{\mu^{3/2}} |u|^{-1}Z_{u/|u|^{1/d}} \leq
		f^+(u),\; u\in K}
	\end{multline*}
	as $t\to\infty$, where $Z_x$, $x\in\Rp$, is the Chentsov random field.
\end{theorem}
\begin{proof}
	By the definition of the radial function, 
	\begin{equation}
	\begin{aligned}
	\label{PPP}
	&\left\{\sqrt{t}\left((r_t(u))^d
		-\frac{1}{\mu|u|}\right)\leq f^+(u), \; u\in K\right\}\\
	=\;&\left\{r_t(u)\le y_t^+(u),\; u\in K\right\}
	=\left\{S_{t^{1/d}y_t^+(u)u}\ge t,\; u\in K\right\},
	\end{aligned}
	\end{equation}
	where
	\begin{displaymath}
	y_t^+(u)=\left(\frac{f^+(u)}{\sqrt{t}}+\frac{1}{\mu |u|}\right)^{1/d}
	=\left(\frac{\mu f^+(u)|u|}{\sqrt{t}}+1\right)^{1/d}(\mu |u|)^{-1/d}.
	\end{displaymath}
	Let $M^+$ be the supremum of $\mu |f^+(u)||u|$ over $u\in K$, and so
	$y_t^+(u)=(\alpha^+)^{1/d}(\mu |u|)^{-1/d}$ with an
	$\alpha^+=\alpha^+(t,u)\in[1-M^+ t^{-1/2},1+M^+ t^{-1/2}]$.
	
	Thus, continuing \eqref{PPP},
	\begin{align*}
	\bigg\{\sqrt{t}\Big((r_t(u))^d
	&-\frac{1}{\mu|u|}\Big)\leq f^+(u), \; u\in K\bigg\}\\
	&=\left\{\frac{S_{t^{1/d}y_t^+(u)u}-|t^{1/d}y_t^+(u)u|\mu}{\sigma\sqrt{t}}
		\geq -\frac{\mu}{\sigma} f^+(u)|u|,\;  u\in K\right\}\\
	&=\left\{\frac{S_{(t\alpha^+/\mu|u|)^{1/d}u}-
			|(t\alpha^+/\mu|u|)^{1/d}u|\mu}{\sigma\sqrt{t}}
		\geq -\frac{\mu}{\sigma} f^+(u)|u|,\;  u\in K\right\}\\
	&=\left\{\tilde S_{(t/\mu)^{1/d},(\alpha^+)^{1/d}u/|u|^{1/d}}\geq
		-\frac{\mu^{3/2}}{\sigma} f^+(u)|u|,\; u\in K\right\}
	\end{align*}
	with $\tilde S$ defined by \eqref{int_field}.
	It follows from the above equality and its counterpart for $f^-$ that
	\begin{align*}
	&\Prob{f^-(u)<\sqrt{t}\left((r_t(u))^d-\frac{1}{\mu|u|}\right)\leq f^+(u),\;
	u\in K}\\=\,&\Prob{\tilde S_{(t/\mu)^{1/d},(\alpha^+)^{1/d}u/|u|^{1/d}}\geq
	-\frac{\mu^{3/2}}{\sigma} f^+(u)|u|,\; \tilde S_{(t/\mu)^{1/d},(\alpha^-)^{1/d}u/|u|^{1/d}}<
	-\frac{\mu^{3/2}}{\sigma} f^-(u)|u|,\; u\in K}.
	\end{align*}
	
	Note that
	\begin{displaymath}
	\left(\tilde S_{(t/\mu)^{1/d},(\alpha^+)^{1/d}u/|u|^{1/d}},\;
	\tilde S_{(t/\mu)^{1/d},(\alpha^-)^{1/d}u/|u|^{1/d}},\;u\in K\right)\to
	\left(Z_{u/|u|^{1/d}},\;Z_{u/|u|^{1/d}},\;u\in K\right)
	\end{displaymath}
	weakly in the uniform metric as $t\to\infty$, since $\alpha^\pm(t,u)\to 1$ uniformly over $u\in K$.  It remains to use the symmetry
	property of the Chentsov random field.
\end{proof}

\begin{remark} 
	The random field
	\begin{displaymath}
	\zeta_u=|u|^{-1}Z_{u/|u|^{1/d}}, \quad u\in\Rp,
	\end{displaymath}	
	which coincides with the limiting field up to a constant, has the
	covariance
	\begin{displaymath}
	\E(\zeta_u\zeta_v)=\frac{\left||u|^{1/d}v\wedge|v|^{1/d}u\right|}{(|u||v|)^2},
	\end{displaymath}
	which becomes $|u\wedge v|$ if $|u|=|v|=1$.  Since
	\begin{equation}
	\label{self-sim}
	\zeta_{cu}=c^{-d}\zeta_u\quad\text{for any $c>0$,}
	\end{equation}
	and $\zeta_u=Z_u$ if $|u|=1$, $\zeta_u$ can be obtained by
	extrapolation of $Z_u$ from $\{u\in\Rp:\; |u|=1\}$ to $\Rp$ by means
	of \eqref{self-sim}.
\end{remark}

\section{Proofs for results in Section~\ref{sec:strong-law-large}}
\label{sec:proofs-results-sect}

Since we have to prove inclusions (\ref{incl}) only for large $t$, the
function $p$ may be arbitrarily redefined in a neighbourhood of the
origin. Particularly, we may assume that $p(t)$ is positive and
non-decreasing for all $t\ge0$, and $t^{-1}p(t)$ is non-increasing for
all $t>0$.

First, list some immediate properties of the function $p$ needed in
the sequel.

\begin{lemma}
  \label{lem}
  Under the assumptions of Theorem~\ref{inv_th}, 
  \begin{enumerate}[(a)]
  \item $p(t)=\sO(t)$ as $t\to\infty$;
  \item $\liminf_{t\to\infty} p(t)/p(ct+\delta p(t))>0$ for any $c,\delta>0$;
  \item $p(t)-\delta t$ is non-increasing in $t$ for large $\delta$ and $t$.
  \end{enumerate}	
\end{lemma}
\begin{proof}
  (a) follows from the fact that $p(t)/t$ is non-increasing 
  due to the positivity of $p$.
  By (a), taking into account that $p(t)$ is non-decreasing, 
  \begin{displaymath}
    \liminf_{t\to\infty}\frac{p(t)}{p(ct+\delta p(t))}\ge
    \liminf_{t\to\infty}\frac{p(t)}{p(ct+\delta Mt)}
  \end{displaymath}
  with some $M>0$, the right-hand side being positive due to
  (\ref{ORV_inf}).
  
  Since $p(t)/t$ is non-increasing, $\delta-t^{-1}p(t)$ is positive
  and non-decreasing in $t$ for large $\delta$ and $t$. Hence,
  \begin{displaymath}
    p(t)-\delta t=-t(\delta-t^{-1}p(t))
  \end{displaymath}
  does not increase.
\end{proof}

Next, we show that the asymptotic behaviour of $S_\n$ given by
(\ref{hyp}) is inherited by the interpolated sums. 

\begin{lemma}
  \label{lem:int-sums}
  If (\ref{hyp}) holds, then 
  \begin{equation}
    \label{hyp_x}
    S_\x-\mu|\x|=\ssO(p(|\x|))\quad\text{a.s. as }\; \x\to\infty.
  \end{equation}
\end{lemma}
\begin{proof}
  Being multi-linear itself, $|\x|$ can be exactly recovered by
  \begin{equation}
    \label{int_abs}
    |\x|=\sum_{\k\in C_\x}v_\k(\x)\,|\k^\ast|,\quad \x\in\R_+^d.
  \end{equation}
  Let $\widetilde C_\x=\{\k\in C_\x:|\k^\ast|\ne0\}$. By (\ref{int}),
  (\ref{int_abs}), and monotonicity of $p$, we have for all $\x\in\R_+^d$
  \begin{align*}
    \frac{|S_\x-\mu|\x||}{p(|\x|)}&\le
    \frac{\sum_{\k\in \widetilde C_\x}
      v_\k(\x)|S_{\k^\ast}-\mu|\k^\ast||}
    {p(\sum_{\k\in \widetilde C_\x}v_\k(\x)\,|\k^\ast|)}
    \\
    &=\sum_{\k\in \widetilde C_\x}
    \frac{v_\k(\x)|S_{\k^\ast}-\mu|\k^\ast||}
    {p(v_\k(\x)\,|\k^\ast|)}\;
    \frac{p(v_\k(\x)\,|\k^\ast|)}
    {p(\sum_{\k\in \widetilde C_\x}v_\k(\x)\,|\k^\ast|)}\\
    &\le
    \sum_{\k\in \widetilde C_\x}
    \frac{v_\k(\x)|S_{\k^\ast}-\mu|\k^\ast||}
    {p(v_\k(\x)\,|\k^\ast|)}\\ 
    &=
    \sum_{\k\in \widetilde C_\x}
    \frac{|S_{\k^\ast}-\mu|\k^\ast||}
    {p(|\k^\ast|)}\;\frac{p(|\k^\ast|)}
    {|\k^\ast|}\;\frac{v_\k(\x)\,|\k^\ast|}
    {p(v_\k(\x)\,|\k^\ast|)}.
  \end{align*}
  Since $p(t)/t$ is non-increasing, 
  \begin{equation*}
    \frac{|S_\x-\mu|\x||}{p(|\x|)}
    \le\sum_{\k\in \widetilde C_\x}
    \frac{|S_{\k^\ast}-\mu|\k^\ast||}{p(|\k^\ast|)},
  \end{equation*}
  and so (\ref{hyp}) implies (\ref{hyp_x}).
\end{proof}

\begin{proof}[Proof of Theorem~\ref{inv_th}]
  Assume that the left-hand inclusion in (\ref{incl}) does not hold,
  that is, there are sequences $\{\x_i,i\geq1\}$ and $\{t_i,i\geq1\}$
  with $t_i\to\infty$, such that $\x_i\in\sH(\eps p(t_i)t_i^{-1})$ and
  $\x_i\notin t_i^{-1/d}\sM_{t_i}$ for all $i$.  Denoting
  $\y_i=t_i^{1/d}\x_i$, we may write the former inclusion as
  $|\y_i|\ge\mu^{-1}t_i+\eps p(t_i)$ and the latter one as
  $S_{\y_i}<t_i$. The first inequality implies $\y_i\to\infty$.
  Hence,
  \begin{align*}
    \alpha_i&=\frac{S_{\y_i}-\mu|\y_i|}{p(|\y_i|)}=
    \frac{|\y_i|}{p(|\y_i|)}\left(\frac{S_{\y_i}}{|\y_i|}-\mu\right)\\
    &<\frac{|\y_i|}{p(|\y_i|)}\left(\frac{t_i}{\mu^{-1}t_i+\eps p(t_i)}-\mu\right)
    =-\;\frac{|\y_i|}{p(|\y_i|)}\frac{\eps\mu p(t_i)}
    {\mu^{-1}t_i+\eps p(t_i)}.
  \end{align*}
  Since $p(t)/t$ is non-increasing, 
  \begin{equation}
    \label{inq}
    \alpha_i<-\;\frac{\mu^{-1}t_i+\eps p(t_i)}{p(\mu^{-1}t_i+\eps p(t_i))}\;
    \frac{\eps\mu p(t_i)}{\mu^{-1}t_i+\eps p(t_i)}
    =-\;\frac{\eps\mu p(t_i)}{p(\mu^{-1}t_i+\eps p(t_i))}.
  \end{equation}
  Note that $\alpha_i\to 0$ as $i\to\infty$ by (\ref{hyp_x}), whereas
  the negative right-hand side of (\ref{inq}) is bounded away from
  zero by Lemma~\ref{lem}(b). This contradiction proves
  the left-hand inclusion in (\ref{incl}).

  Assume that the right-hand inclusion in (\ref{incl}) does not hold,
  so that there exist sequences $\{\x_i,i\geq1\}$ and
  $\{t_i,i\geq1\}$ with $t_i\to\infty$ such that $
  |\y_i|<\mu^{-1}t_i-\eps p(t_i)$ and $S_{\y_i}\ge t_i$ for all $i$,
  where $\y_i=t_i^{1/d}\x_i$. Therefore, $S_{\y_i}\to\infty$, which
  easily leads to $\y_i\to\infty$ by (\ref{int}). 
  By Lemma~\ref{lem}(c), 
  \begin{equation}
    \label{y_i2}
    |\y_i|<\mu^{-1}S_{\y_i}-\eps p(S_{\y_i})
  \end{equation}
  for large $i$ and sufficiently small $\eps>0$ (that may be smaller
  than the first chosen $\eps$). Using the above definition of
  $\alpha_i$, we get
  \begin{align*}
    \alpha_i&=
    \frac{|\y_i|}{p(|\y_i|)}\left(\frac{S_{\y_i}}{|\y_i|}-\mu\right)\\
    &>\frac{|\y_i|}{p(|\y_i|)}\left(\frac{S_{\y_i}}{\mu^{-1}S_{\y_i}-
      \eps p(S_{\y_i})}-\mu\right)
    =\frac{|\y_i|}{\mu^{-1}S_{\y_i}-\eps p(S_{\y_i})}\;
    \frac{\eps\mu p(S_{\y_i})}{p(|\y_i|)}.
  \end{align*}
  By (\ref{y_i2}) and taking into account the monotonicity of $p$, we have
  \begin{equation}
    \label{ineq}
    \alpha_i>\frac{\mu|\y_i|}{S_{\y_i}}\;\frac{\eps\mu p(S_{\y_i})}
    {p(\mu^{-1}S_{\y_i}-\eps p(S_{\y_i}))}\ge\frac{\mu|\y_i|}{S_{\y_i}}
    \;\frac{\eps\mu p(S_{\y_i})}
    {p(\mu^{-1}S_{\y_i})}.
  \end{equation}

  Note that
  \begin{equation}
    \label{hyp_x2}
    S_\y-\mu|\y|=\ssO(|\y|)\quad\text{a.s. as }\;\y\to\infty.
  \end{equation}
  This is not a straightforward consequence of (\ref{hyp_x}) and 
  Lemma~\ref{lem}(a), since $\y\to\infty$ need not imply
  $|\y|\to\infty$ (which is possible if $\y\to\infty$ while getting
  simultaneously closer to one of the coordinate planes). However,
  (\ref{hyp_x2}) may be proved in an alternative way: (\ref{hyp}) and
  (a) lead to $S_\n-\mu|\n|=\ssO(|\n|)$ a.s.\
  as $\n\to\infty$ in $\NN^d$ (which is now equivalent to
  $|\n|\to\infty$), and the latter in turn implies (\ref{hyp_x2}) in
  the same manner as (\ref{hyp}) implies (\ref{hyp_x}).

  So, by (\ref{hyp_x2})
  \begin{equation}
    \label{KSLLN}
    \frac{\mu|\y_i|}{S_{\y_i}}\to 1\quad\text{a.s. as }\; i\to\infty.
  \end{equation}
  At the same time, the second factor on the right-hand side of
  (\ref{ineq}) is bounded away from zero as $i\to\infty$ due to
  (\ref{ORV_inf}). This contradicts $\alpha_i\to 0$ and so
  proves the right-hand inclusion in~(\ref{incl}).
\end{proof}

The following results give bounds on the Hausdorff and the symmetric
difference distances between the sets $\sH(c)$ for different $c$'s.

\begin{lemma}
  \label{rhoH_lem}
  If $-\mu^{-1}<c_1\le c_2$, then
  \begin{equation}
    \label{rhoHH}
    \rhoH(\sH(c_1),\sH(c_2))=\sqrt d\,
    ((\mu^{-1}+c_2)^{1/d}-(\mu^{-1}+c_1)^{1/d}).
  \end{equation}
  If $c_1,c_2\to 0$, then
  \begin{equation}
    \label{rho_H_lim}
    \rhoH(\sH(c_1),\sH(c_2))
    =d^{-1/2}\mu^{1-1/d}(c_2-c_1)+\ssO(c_2-c_1).
  \end{equation}
\end{lemma}
\begin{proof}
  An elementary minimisation argument yields that 
  \begin{displaymath}
    \inf\{\langle u,x\rangle:\; x\in\sH(c)\}=d(c+\mu^{-1})^{1/d}|u|^{1/d}
  \end{displaymath}
  for all $u\in\Sphere\cap\R_+^d$. The above expression yields the
  negative of the support function of $\sH(c)$ in direction
  $(-u)$. Since the Hausdorff distance between convex sets $\sH(c_1)$
  and $\sH(c_2)$ equals the uniform distance between their support
  functions and the maximal value of $|u|$ is $d^{-d/2}$,
  \eqref{rhoHH} holds and easily yields (\ref{rho_H_lim}).
\end{proof}

\begin{lemma}
  \label{rhotr_lem}
  Let $\cone$ be a cone in $\Rp$. If $-\mu^{-1}<c_1\le c_2$, then
  \begin{equation}
    \label{rhotr_bound}
    \rho_\triangle^\cone(\sH(c_1),\sH(c_2))= 
    L_\cone(c_2-c_1),
  \end{equation}
  where $L_\cone$ is given by \eqref{LK}. 
\end{lemma}
\begin{proof}
  Put $b_i(u)=(\mu^{-1}+c_i)^{1/d}|u|^{-1/d}$, $i=1,2$.
  Equation~\eqref{rhotr_bound} easily follows by representing
  $\cone\cap(\sH(c_1)\setminus\sH(c_2))$ in the spherical coordinates:
  \begin{equation*}
  \rho_\triangle^\cone(\sH(c_1),\sH(c_2))=\int_{\cone\cap\Sphere}
  \bigg(\int_{b_1(u)}^{b_2(u)}r^{d-1}\,\dif r\bigg)\,\dif u
  =\frac{c_2-c_1}d\int_{\cone\cap\Sphere}|u|^{-1}\,\dif u.\qedhere
  \end{equation*}
\end{proof}

\begin{proof}[Proof of Theorem \ref{met_SLLN_th}] 
  By Corollary \ref{set_SLLN_cor}, \eqref{rhoHH} and (\ref{rho_H_lim}), 
  \begin{displaymath}
    \rhoH(t^{-1/d}\sM_t,\sH) \leq \sqrt{d}((\mu^{-1}+\eps
    t^{-1+1/\beta})^{1/d} -(\mu^{-1}-\eps t^{-1+1/\beta})^{1/d})
    =\eps\sO(t^{-1+1/\beta}).
  \end{displaymath}
  A similar bound for the symmetric difference metric follows from
  \eqref{rhotr_bound}.  Since $\eps$ can be chosen arbitrary small,
  (\ref{H_bound}) and (\ref{D_bound}) follow.  If $K\subset\Rp$, then
  $K$ is a subset of a cone $\cone$ with
  $\cone\setminus\{0\}\subset\Rp$, so that
  Corollary~\ref{cor:sector-SLLN} applies.
\end{proof}

\section{Proofs for results in Section~\ref{sec:law-iter-logar}}
\label{sec:proofs-results-sect-2}

It suffices to assume that $\sigma=1$.  To simplify the notation, let
\begin{equation}
  \label{p}
  \chi(t)=\sqrt{2t\log\log t}=t\fnc(t)
\end{equation}
for $t\ge e^e$, and extend both $\chi$ and $\fnc$ to $[0,\infty)$ and
$(0,\infty)$, respectively, so that $\chi$ becomes positive and
concave.

It follows from the law of the iterated logarithm for multi-indexed
sums due to Wichura \cite[Th.~5]{Wich}, see also
\cite[Th.~10.9]{KlesBook}, that, under \eqref{Wich_cond},
\begin{equation*}
  \limsup_{\n\to\infty}\frac{|S_\n-\mu|\n||}{\chi(|\n|)}=\sqrt{d} 
  \quad\text{a.s.}
\end{equation*}
Hence,
\begin{equation}
  \label{ord_LIL_x}
  \limsup_{\x\to\infty}\frac{|S_\x-\mu|\x||}{\chi(|\x|)}=\sqrt{d}
  \quad\text{a.s.}
\end{equation}
Indeed, 
\begin{align*}
  \frac{|S_\x-\mu|\x||}{\chi(|\x|)}&\le
  \frac{\sum_{\k\in C_\x}
    v_\k(\x)|S_{\k^\ast}-\mu|\k^\ast||}
  {\chi(\sum_{\k\in C_\x}v_\k(\x)\,|\k^\ast|)}\\
  &\le\frac{\sum_{\k\in C_\x}	v_\k(\x)|S_{\k^\ast}-\mu|\k^\ast||}
  {\sum_{\k\in C_\x}v_\k(\x)\,\chi(|\k^\ast|)}\le
  \max_{\k\in C_\x}
  \frac{|S_{\k^\ast}-\mu|\k^\ast||}
  {\chi(|\k^\ast|)},
\end{align*}
where the second inequality relies on the concavity of $\chi$.  The
same argument applied to the sectorial version of the LIL proved in
Theorem~\ref{S_set_SLLN_lem} leads to
\begin{equation}
  \label{sect_LIL_x}
  \limsup_{\cone\ni\x\to\infty}\frac{|S_\x-\mu|\x||}{\chi(|\x|)}=1
  \quad\text{a.s.}
\end{equation}

\begin{proof}[Proof of Theorem \ref{LIL}]
  First, we prove the inclusion in (i). Taking (\ref{Hast}) into
  account, we actually need to show that 
  \begin{equation}
    \label{incls1}
    t^{-1/d}\sM_t\subset\sH(\gamma\sqrt{d}\fnc(t))
  \end{equation}
  and 
  \begin{equation}
    \label{incls2}
    t^{-1/d}(\cone\cap\sM_t)\subset
    (\cone\,\cap\,\sH(\gamma\fnc(t))
  \end{equation}
  almost surely for all sufficiently large $t$.

  In order to derive (\ref{incls1}), we assume the contrary and
  consider the sequences $\{\y_i,i\geq1\}$ and
  $\{t_i,i\geq1\}$ with $\y_i,t_i\to\infty$ such that 
  \begin{displaymath}
    |\y_i|<\mu^{-1}t_i+\gamma \sqrt{d}\chi(t_i)
  \end{displaymath}
  and $S_{\y_i}\ge t_i$ for all $i$. Along the
  lines of the proof of Theorem~\ref{inv_th} (with $-\gamma$ instead
  of $\eps$ and $\sqrt d\chi(\cdot)$ instead of $p(\cdot)$), we arrive
  at an analogue of inequality (\ref{ineq}):
  \begin{equation*}
    \alpha_i=\frac{S_{\y_i}-\mu|\y_i|}{\sqrt{d}\chi(|\y_i|)}
    >-\;\frac{\mu|\y_i|}{S_{\y_i}}
    \;\frac{\gamma\mu \chi(S_{\y_i})}
    {\chi(\mu^{-1}S_{\y_i})}.
  \end{equation*} 

  Passing to the upper limit, by (\ref{ord_LIL_x}), (\ref{KSLLN}), and
  (\ref{p}) we arrive at the contradiction
  \begin{equation*}
    1\ge\limsup_{i\to\infty}\alpha_i\ge-\gamma\mu^{3/2}>1.
  \end{equation*}
  The same argument with $\y_i\in\cone$ and a reference to
  \eqref{sect_LIL_x} leads to~(\ref{incls2}).

  The inclusion in (iii) may be deduced in a similar manner by means of
  (\ref{inq}) instead of (\ref{ineq}) and $\liminf$ instead of
  $\limsup$.

  Let us now turn to the proof of (ii). Since $\sH_\cone(c)$ decreases
  in $c$, it suffices to prove that (\ref{notincl1}) and
  (\ref{notincl2}) hold with $\gamma=-\mu^{-3/2}$ and
  $\gamma=\mu^{-3/2}$, respectively. It will be shown that
  ``exceptional'' points which violate these inclusions may be found
  on the diagonal
  \begin{displaymath}
    \sD=\{\x\in\Rp:\; x^1=\cdots=x^d\}.
  \end{displaymath}
  This, however, requires a more delicate analysis.
  Introduce the sequence of diagonal integer points
  \begin{displaymath}
    \sD\ni\z_i=i\cdot\1=(i,\dots,i),\qquad i\geq1,
  \end{displaymath}
  and a (one-dimensional) sequence $\{\eta_j,j\geq1\}$ of independent
  copies of $\xi$. For $i\geq1$, denote $\tilde
  S_i=\sum_{j=1}^i\eta_j$.  By \cite[Th.~1.1]{Berkes} (see also (1.14)
  ibid.), it may be easily checked that, under assumption
  \begin{equation}
    \label{BWC}
    \E(\xi^2\log\log|\xi|)<\infty,
  \end{equation}
  which holds by (\ref{Wich_cond}),
  \begin{equation}
    \label{q}
    q(t)=\sqrt{2t(\log\log t+1)},\quad t\ge0,
  \end{equation}
  is a lower function for $\{\tilde S_{i^d},i\geq1\}$. Hence,
  $q$ is a lower function
  for the sequence $\{S_{\z_i},i\geq1\}$, which has the same
  distribution. In other words, each of the inequalities
  \begin{equation}
    \label{io}
    S_{\z_i}\le\mu|\z_i|-q(|\z_i|),
    \quad S_{\z_i}\ge\mu|\z_i|+q(|\z_i|)
  \end{equation}
  holds infinitely often with probability one.

  In order to prove the claim, it suffices to find (random) sequences
  $\{t'_i,i\geq1\}$ and $\{t''_i,i\geq1\}$ such that
  $t'_i,t''_i\to\infty$ a.s., and for large $i$ a.s.
  \begin{align*}
    (\sD\cap\sM_{t_i'})&\not\subset\sD\cap
    \left\{\y\in\Rp:\;|\y|\ge\mu^{-1}t_i'-
    \mu^{-3/2}\chi(t_i')\right\},\\
    (\sD\cap\sM_{t_i''})&\not\supset\sD\cap
    \left\{\y\in\Rp:\;|\y|\ge\mu^{-1}t_i''+
    \mu^{-3/2}\chi(t_i'')\right\}.
  \end{align*}
  Following (\ref{io}), we introduce (random) sequences of indices
  $\{\z'_i,i\geq1\}$ and $\{\z''_i,i\geq1\}$ such that
  $\z'_i,\z''_i\in\NN^d\cap\sD$, $\z'_i,\z''_i\to\infty$ a.s., and
  \begin{displaymath}
    S_{\z'_i}\ge\mu|\z'_i|+q(|\z'_i|),\qquad
    S_{\z''_i}\le\mu|\z''_i|- q(|\z''_i|)
  \end{displaymath}
  almost surely for all sufficiently large $i$.
  Letting $t'_i=S_{\z'_i}$ and $t''_i=S_{\z''_i}+1$ yields that
  $\z'_i\in\sM_{t'_i}$ and $\z''_i\notin\sM_{t''_i}$. Hence, we
  actually need to prove that the implications 
  \begin{align}
    \label{impl2}
    S_{\z'_i}\ge\mu|\z'_i|+q(|\z'_i|)&\Rightarrow|\z'_i|<
    \mu^{-1}S_{\z'_i}-\mu^{-3/2}\chi(S_{\z'_i}),\\
    \label{impl3}
    S_{\z''_i}\le\mu|\z''_i|-q(|\z''_i|)&\Rightarrow|\z''_i|\ge
    \mu^{-1}(S_{\z''_i}+1)+\mu^{-3/2}\chi(S_{\z''_i}+1)
  \end{align}
  hold a.s.~for all sufficiently large $i$.  Setting $\psi_-(u)=\mu
  u-q(u)$, $\psi_+(u)=\mu u+q(u)$, and denoting by
  $\psi_-^{\leftarrow}$ and $\psi_+^{\leftarrow}$ their inverses, we
  may write the left-hand inequalities in (\ref{impl2}) and
  (\ref{impl3}) as $|\z'_i|\le\psi_+^{\leftarrow}(S_{\z'_i})$ and
  $|\z''_i|\ge\psi_-^{\leftarrow}(S_{\z''_i})$. Thus, it suffices to
  show that the inequalities
  \begin{align*}
    \psi_+^{\leftarrow}(u)&<\mu^{-1}u-\mu^{-3/2}\chi(u),\\
    \psi_-^{\leftarrow}(u)&\ge\mu^{-1}(u+1)+\mu^{-3/2}\chi(u+1)
  \end{align*}
  hold for large $u$. A straightforward calculation yields that
  these inequalities actually mean
  \begin{align*}
    q(\mu^{-1}u-\mu^{-3/2}\chi(u))&>\mu^{-1/2}\chi(u),\\
    q(\mu^{-1}(u+1)+\mu^{-3/2}\chi(u+1))&\ge
    \mu^{-1/2}\chi(u+1)+1.
  \end{align*}
  Routine but rather tedious calculations (which we do not detail
  here) show that the above inequalities indeed hold for large $u$
  with $\chi$ and $q$ defined by (\ref{p}) and (\ref{q}). This
  completes the proof of (ii) and of Theorem~\ref{LIL}.
\end{proof}

\begin{remark}
  \label{SectMomCond}
  The sectorial LIL proved in Theorem~\ref{sectLILlem} does not
  require Wichura's condition (\ref{Wich_cond}). Hence, all parts of
  the foregoing proof based only on sectorial arguments remain true
  without (\ref{Wich_cond}). This particularly applies to
  (\ref{incls2}) with $\gamma<-\mu^{-3/2}$ as well as to the reverse
  inclusion with $\gamma>\mu^{-3/2}$. For ease of reference, we
  reproduce them here in a slightly modified form
  \begin{displaymath}
  	\cone\cap\sH(\gamma\fnc(t))\subset
    t^{-1/d}(\cone\cap\sM_t)\subset \cone\cap\sH(-\gamma\fnc(t))
  \end{displaymath}
  a.s.~for $\gamma>\mu^{-3/2}$ and all sufficiently large $t$.
\end{remark}

\begin{proof}[Proof of Theorem \ref{met_LIL_th}]
  Fix $\gamma>\mu^{-3/2}$ and a closed convex cone $\cone$ with
  $\cone\setminus\{0\}\subset\Rp$. Denote for brevity 
  $\sH^\pm=\sH(\pm\gamma\fnc(t))$ and
  $\sH^\pm_\cone=\sH_\cone(\pm\gamma\fnc(t))$. 
  By (i) and (iii) in Theorem~\ref{LIL},
  \begin{equation}
    \label{M_incl}
    \sH^+_\cone\subset t^{-1/d}\sM_t\subset\sH^-_\cone
  \end{equation}
  almost surely for all sufficiently large $t$. Therefore, 
  \begin{displaymath}
    \rhoH(t^{-1/d}\sM_t,\sH)\le
    \max\{\rhoH(\sH,\sH^-_\cone),\rhoH(\sH,\sH^+_\cone)\}
  \end{displaymath}
  for all sufficiently large $t$.

  Without loss of generality, assume that $\cone$ is sufficiently
  large and contains the diagonal, so that
  $\rhoH(\sH,\sH^\pm_\cone)=\rhoH(\sH,\sH^\pm)$. By (\ref{rho_H_lim}),
  \begin{displaymath}
    \rhoH(t^{-1/d}\sM_t,\sH)\le
    d^{-1/2}\gamma\mu^{1-1/d}\fnc(t)+\ssO(\fnc(t))
    \quad \text{a.s.~as }\; t\to\infty.
  \end{displaymath}
  Dividing by $\fnc(t)$ and letting $\gamma\downarrow\mu^{-3/2}$
  yields the upper bound in (\ref{met_LIL_H}):
  \begin{displaymath}
    \limsup_{t\to\infty}\frac{\rhoH(t^{-1/d}\sM_t,\sH)}
    {\fnc(t)}\le d^{-1/2}\mu^{-1/2-1/d}\quad\text{a.s.}
  \end{displaymath}

  In order to obtain the reverse inequality, we notice that the
  sequences $\{\z_i',i\geq1\}$ and $\{t_i',i\geq1\}$ with
  $t_i'=S_{\z_i'}$ constructed in the final part of the proof of
  Theorem~\ref{LIL} a.s.~satisfy
  \begin{displaymath}
    (t_i')^{-1/d}\z_i'\in(t_i')^{-1/d}\sM_{t_i'},\quad 
    (t_i')^{-1/d}\z_i'\notin\sH(-\mu^{-3/2}\fnc(t_i'))
  \end{displaymath}
  for large $i$. Since the supremum in the definition of
  $\rhoH(\sH(c_1),\sH(c_2))$ is attained at a diagonal point,
  (\ref{rho_H_lim}) implies
  \begin{align*}
    \frac{\rhoH((t_i')^{-1/d}\sM_{t_i'},\sH)}{\fnc(t_i')}
    \ge\frac{\inf_{\y\in\sH}\rho((t_i')^{-1/d}\z_i',\y)}{\fnc(t_i')}
    &>\frac{\rhoH\left(\sH(-\mu^{-3/2}\fnc(t_i'),\sH\right)}{\fnc(t_i')}\\
    &=d^{-1/2}\mu^{-1/2-1/d}+\ssO(1)\quad \text{as }\; i\to\infty.
  \end{align*}
  Thus, we arrive at the lower bound in (\ref{met_LIL_H}):
  \begin{displaymath}
    \limsup_{t\to\infty}\frac{\rhoH(t^{-1/d}\sM_t,\sH)}{\fnc(t)}
    \ge d^{-1/2}\mu^{-1/2-1/d}\quad\text{a.s.}
  \end{displaymath}

  Let us now turn to the proof of (\ref{met_LIL_tr}). Consider an
  enlarged closed convex cone $\ekcone$ such that
  $\ekcone\setminus\{0\}\subset\Rp$ and whose interior contains
  $\kcone\setminus\{0\}$. Notice that
  \begin{displaymath}
  \rho_\triangle^\ekcone(\sH^-_\ekcone,\sH^+_\ekcone)=
  \rho_\triangle^\ekcone(\sH^-,\sH^+),
  \end{displaymath}
  since $\sH^\pm_\ekcone$ coincides with $\sH^\pm$ within $\ekcone$.  
  Hence, by (\ref{M_incl}) and (\ref{rhotr_bound}),
  \begin{equation}
    \label{rho_ineq}
    \rho_\triangle^K(t^{-1/d}\sM_t,\sH)\le
    \rho_\triangle^\ekcone(\sH^-,\sH^+)
    =2\gamma L_\ekcone\fnc(t) 
  \end{equation}
  almost surely for all sufficiently large $t$.
  Dividing by $\fnc(t)$ and letting first $t\to\infty$ and then
  $\gamma\downarrow\mu^{-3/2}$ and $\ekcone\downarrow\kcone$ yield
  (\ref{met_LIL_tr}).

  Let now $K\subset\Rp$. Choose a cone $\cone$, so that
  $K\subset\cone\setminus\{0\}\subset\Rp$. By Remark \ref{SectMomCond},
  \begin{equation}
  \label{cone_incl}
  K\cap\sH^+\subset K\cap(t^{-1/d}\sM_t)\subset K\cap\sH^-
  \end{equation}
  for all large $t$, provided only that $\E\xi^2<\infty$.
  The rest of the proof follows the lines of the preceding proof, but with reference to \eqref{cone_incl} instead of \eqref{M_incl}.
  
  Assume that $\xi$ is almost surely non-negative.  Then, with each
  $x$, the set $\sM_t$ contains also $ax$ for all $a\geq1$. Hence,
  reflecting the set $t^{-1/d}\sM_t\setminus\sH$ symmetrically with
  respect to $\partial\sH$ in the radial direction, we easily arrive
  at the counterpart of \eqref{rho_ineq}:
  \begin{align*}
    \rho_\triangle^K(t^{-1/d}\sM_t,\sH)\le
    \rho_\triangle^\ekcone(\sH,\sH^+)
    =\gamma L_\ekcone\fnc(t)
  \end{align*}
  almost surely for all sufficiently large $t$,
  and then the proof proceeds as above.  The case $K\subset\Rp$ is
  treated in the same way as before.
\end{proof}

\begin{proof}[Proof of Theorem \ref{liminf_th}] 
  Fix a sufficiently large closed convex cone $\cone$ such that
  $\cone\setminus\{0\}\subset\Rp$, and put $F=\cone\cap\Sphere$. For
  $l\in\NN$ and $c>0$, let
  \begin{displaymath}
    A_{l,c}=\{\omega:\; \cone\cap\sH(cl^{-1/2})\subset\cone\cap(l^{-1/d}\sM_l)
    \subset\cone\cap\sH(-cl^{-1/2})\}.
  \end{displaymath}
  The event $A_{l,c}$ means that, inside $\cone$, the boundary of
  $l^{-1/d}\sM_l$ lies within a relatively narrow strip
  $\mu^{-1}-cl^{-1/2}\le|\x|\le\mu^{-1}+cl^{-1/2}$.

  Let $R_{l,c}^\pm(u)$, $u\in F$, be the radial functions of $\sH(\pm
  cl^{-1/2})$, that is, 
  \begin{displaymath}
    R_{l,c}^\pm(u)=\inf\{a>0:\; au\in\sH(\pm cl^{-1/2})\}=
    \left(\frac{\mu^{-1}\pm cl^{-1/2}}{|u|}\right)^{1/d}.
  \end{displaymath}
  Thus,
  \begin{align*}
    A_{l,c}&=\{\omega:\; R_{l,c}^-(u)\le r_l(u)\le
    R_{l,c}^+(u),\;u\in F\}\\
    &\supset\{\omega\colon R_{l,c}^-(u)<r_l(u)\le R_{l,c}^+(u),\;u\in F\}
  \end{align*}
  and the latter event is identical to 
  \begin{equation}
    \label{Blc}
    B_{l,c}=\left\{\omega:\; 
      -\frac c{|u|}<\sqrt{l}\left((r_l(u))^d-\frac{1}{\mu|u|}\right)
      \le\frac c{|u|},\;u\in F\right\}.
  \end{equation}
  Since $\xi$ is a.s.~non-negative, $B_{l,c}$ can be represented in
  terms of interpolated sums as
  \begin{equation}
    \begin{aligned}
      \label{Blc_int}
      B_{l,c}=\Big\{\omega\colon &S_{l^{1/d}x}<l
      \text{ for all $x\in\cone$ with $|x|=\mu^{-1}-cl^{-1/2}$,}\\
      &S_{l^{1/d}x}\ge l\text{ for all $x\in\cone$ with $|x|=\mu^{-1}+cl^{-1/2}$}\Big\}.
    \end{aligned}
  \end{equation}
  By \eqref{Blc} and Theorem \ref{thr:radial},
  \begin{align*}
    \lim_{l\to\infty}\P(B_{l,c})
    &=\Prob{-c|u|^{-1}<\frac{\sigma}{\mu^{3/2}} |u|^{-1}Z_{u/|u|^{1/d}}\le
      c|u|^{-1},\; u\in F}\\
    &\ge\Prob{-\frac{c\mu^{3/2}}\sigma<Z_{u/|u|^{1/d}}<\frac{c\mu^{3/2}}\sigma,
      \;u\in F}.
  \end{align*}

  It follows from general results on Gaussian measures in Banach
  spaces that the probability on the right-hand side is positive for
  any $c>0$. For instance, this may be easily deduced from the
  infinite-dimensional Anderson inequality, see, e.g.,
  \cite[Cor.~7.1]{Lif}.  Hence, $\lim_{l\to\infty}\P(B_{l,c})>0$ for
  any $c>0$, and
  \begin{equation}
    \label{P_ineq}
    \Prob{B_{l,c}\text{ i.o.}}=
    \lim_{l\to\infty}\P(\cup_{i\ge l}B_{i,c})\ge
    \lim_{l\to\infty}\P(B_{l,c})>0,
  \end{equation}
  where i.o.~stands for ``infinitely often''.

  It follows from \eqref{Blc_int} that $B_{l,c}$ is measurable with
  respect to the $\sigma$-algebra generated by $S_{l^{1/d}x}$,
  $x\in\cone\cap\sH(-cl^{-1/2})$. So, the random event
  $\{B_{l,c}\text{ i.o.}\}$ is invariant under finite permutations of
  $\NN^d$. Let $e:\NN^d\mapsto\NN$ be the usual zigzag enumeration of
  $\NN^d$.  Applying the Hewitt-Savage $0$-$1$ law to the
  (one-dimensional) sequence $\{\xi_{e(m)},m\in\NN^d\}$ turns
  \eqref{P_ineq} into $\Prob{B_{l,c}\text{ i.o.}}=1$. Hence,
  $\Prob{A_{l,c}\text{ i.o.}}=1$.

  So, Lemmas~\ref{rhoH_lem} and~\ref{rhotr_lem} imply
  \begin{align}
    \label{infcone_1}
    \liminf_{t\to\infty}\sqrt t\,\rhoH(\cone\cap t^{-1/d}\sM_t,\cone\cap\sH)
    &\le 2cd^{-1/2}\mu^{1-1/d}\quad\text{a.s.},\\
    \label{infcone_2}
    \liminf_{t\to\infty}\sqrt t\,\rho_\triangle^K(\cone\cap t^{-1/d}\sM_t,
    \cone\cap \sH)&\le 2cL_\cone\quad\text{a.s.}
  \end{align}
  Under \eqref{Wich_cond}, it follows from \eqref{incls1} and the
  reverse inclusion that
  \begin{displaymath}
    \sH(\gamma\sqrt{d}\fnc(t))\subset t^{-1/d}\sM_t\subset\sH(-\gamma\sqrt{d}\fnc(t))
  \end{displaymath}
  holds for any $\gamma>\mu^{-3/2}$ and all large $t$. By choosing a
  sufficiently large cone $\cone$, we can make
  $\sH(\gamma\sqrt{d}\fnc(t))$ and $\sH(-\gamma\sqrt{d}\fnc(t))$
  arbitrarily close to each other outside $\cone$. Hence,
  $\liminf_{t\to\infty}\sqrt t\,\rhoH(t^{-1/d}\sM_t,\sH)$ is
  determined by the left-hand side of \eqref{infcone_1}, and letting
  $c\to0$ delivers \eqref{liminf_H}.

  The proof of \eqref{liminf_tr} proceeds similarly to that of
  \eqref{met_LIL_tr}, but with reference to (\ref{M_incl}) replaced by
  that to
  \begin{equation*}
    \widehat\sH^+_\cone\subset t^{-1/d}\sM_t\subset\widehat\sH^-_\cone
  \end{equation*}
  with
  \begin{equation*}
    \widehat\sH^\pm_\cone=(\cone\cap\sH(\pm ct^{-1/2}))\cup
    ((\Rp\setminus\cone)\cap\sH(\pm\gamma\sqrt d\fnc(t)))
  \end{equation*}
  and any $\gamma>\mu^{-3/2}$. Letting $c\to0$ completes the proof of
  \eqref{liminf_tr}. Finally, if $K\subset\Rp$ then the claim
  immediately follows from \eqref{infcone_2} by choosing $\cone\supset
  K$ and $c\to0$.
\end{proof}

\section{The one-dimensional case}
\label{sec:one-dimensional-case}

Let us now briefly discuss the case of $d=1$. Then
\begin{equation*}
    \sH(c)=[0,\infty)\cap[\mu^{-1}+c,\infty),
\end{equation*}
and there is no need to introduce the cone $\cone$.
The multidimensional inversion theorem (Theorem~\ref{inv_th}) and the
set-inclusion SLLN (Corollary~\ref{set_SLLN_cor}), together with their
proofs, remain valid in this case, too.  The set-inclusion LIL
(Theorem \ref{LIL}) in the above form additionally requires that
$\E(\xi^2\log\log|\xi|)<\infty$ (see (\ref{BWC}) above), which
in the multidimensional case follows from Wichura's
condition (\ref{Wich_cond}). Under this assumption, which goes back to
Feller, we could apply a Kolmogorov--Petrovsky--Erd\H{o}s--Feller type
criterion in order to check whether a given function is upper or lower
in the LIL for subsequences. 

However, in the one-dimensional setting, this assumption actually
affects only the behaviour at the critical values $\pm\mu^{-3/2}$.
Indeed, if $|\gamma|>\mu^{-3/2}$ (parts (i) and (iii) in Theorem
\ref{LIL}), the above proofs remain valid. In the case of
$-\mu^{-\frac32}<\gamma<\mu^{-\frac32}$, the claim can be proved in
the following alternative way which does not require (\ref{BWC}).

According to the ordinary LIL, there is a (random) sequence of indices
$\{\n_k,k\geq1\}$, such that $\n_k\to\infty$ a.s. and
\begin{equation}
  \label{to1}
  \lim_{k\to\infty}\frac{S_{\n_k}-\mu n_k}{\sigma\chi(\n_k)}=1\quad\text{a.s.}
\end{equation}
Suppose (\ref{notincl1}) does not hold, and so
$t^{-1}\sM_t\subset\sH(\gamma\sigma\fnc(t))$ for all sufficiently
large $t$.  Therefore, $S_\n\ge t$ implies that
$\n\ge\mu^{-1}t+\gamma\sigma\chi(t)$ for all sufficiently large $t$. 
Since $S_{\n_k}\to\infty$ a.s., we may let $\n=\n_k$ and $t=S_{\n_k}$,
so that $\n_k\ge\mu^{-1}S_{\n_k}+\gamma\sigma\chi(S_{\n_k})$. 
By (\ref{to1}), (\ref{p}), and making use of the SLLN for $S_{\n}$, we
arrive at the contradiction
\begin{displaymath}
  1\le-\gamma\mu\lim_{k\to\infty}\frac{\chi(S_{\n_k})}{\chi(\n_k)}=
  -\gamma\mu\lim_{k\to\infty}\sqrt{\frac{S_{\n_k}}{\n_k}}=-\gamma\mu^{\frac32}<1.
\end{displaymath}
Statement (\ref{notincl2}) may be proved in a similar way, noticing
that $S_\n<t$ implies $\n<\mu^{-1}t+\gamma\sigma\chi(t)$, and using
\begin{displaymath}
  \lim_{k\to\infty}\frac{S_{\n_k}-\mu\n_k}{\sigma\chi(\n_k)}=-1\quad\text{a.s.}
\end{displaymath}
instead of (\ref{to1}).  So, Theorem \ref{LIL} remains true in the
one-dimensional case without condition (\ref{BWC}) if
$|\gamma|\ne\mu^{-3/2}$. 

For the metric SLLN and LIL (Theorems \ref{met_SLLN_th} and
\ref{met_LIL_th}) in case $d=1$, one would rather define for $t>0$ the
first passage times
$$\nu(t)=\min\{n\ge1:\; S_n>t\}$$
and the last exit times
$$N(t)=\max\{n\ge0:\; S_n\le t\}.$$
The SLLN and LIL for $\nu(t)$ and $N(t)$ are given in
\cite[Thms.~3.4.4, 3.11.1]{GutBook}.  Note
that the right-hand sides in the cited results are actually identical
to those in (\ref{H_bound}), (\ref{D_bound}) and (\ref{met_LIL_H})
with $d=1$.

Theorem \ref{liminf_th} trivially holds in the one-dimensional case (see the argument above its statement). Theorem \ref{thr:radial} actually reduces in this case to the classical central limit theorem for renewal processes (see, e.g., \cite[Th.~2.5.2]{GutBook}).

\section{Appendix: strong limit theorems for the sectorial convergence}
\label{sec:append-strong-limit}

Fix a closed convex cone $\cone$ with $\cone\setminus\{0\}\subset\Rp$ and denote
\begin{displaymath}
  S_n(\cone)=\sum_{k\in\cone,k\leq n}\xi_k
\end{displaymath}
and 
\begin{displaymath}
  R_n(\cone)=\card\{k\in\cone\cap\NN^d:\; k\leq n\}. 
\end{displaymath}

The a.s.~limit theorems for $S_n(\cone)$ normalised by $R_n(\cone)$
were derived by Gut \cite{GutSect}. Then, lower moment assumptions on
the summands suffice if $n$ converges to infinity inside the
cone. Below we confirm that, with this mode of convergence, the strong
limit theorems hold for $S_n(\cone)$ replaced by $S_n$ and
$R_n(\cone)$ replaced by $|n|$.

\begin{theorem}
  \label{S_set_SLLN_lem}
  If $\E|\xi|^\beta<\infty$ for some $\beta\in[1,2)$, then 
  \begin{equation}
    \label{eq:1}
    S_n-\mu|n|=\ssO(|n|^{1/\beta}) \quad \text{a.s. as }\; \cone\ni n\to\infty. 
  \end{equation}
  \label{sectLILlem}
  If $\E\xi^2<\infty$, then
  \begin{equation}
    \label{mysectLIL}
    \limsup_{\cone\ni n\to\infty}
    \frac{|S_\n-\mu|\n||}
    {\sigma\chi(|n|)}=1 \quad \text{a.s.}
  \end{equation}
\end{theorem}
\begin{proof}
  We will partially apply the approach used in the proofs of
  Theorems~3.1 and~5.1 in \cite{GutSect}.  Fix $m_\cone\in\NN^d$ such
  that all $\x\in\cone$ with $|\x|\leq1$ satisfy $\x\leq m_\cone$,
  that is $m_\cone$ dominates all points from $\{\x\in\cone:\;
  |x|\le1\}$.  The existence of such $m_\cone$ is guaranteed by the
  fact that $\cone\setminus\{0\}$ is a subset of $\Rp$.
  
  We may clearly assume that $\mu=0$ and, in the proof of
  \eqref{mysectLIL}, that $\sigma=1$. Define
  \begin{displaymath}
    A(i)=\{n\in\NN^d\cap\cone:\;
    2^{d(i-1)}\leq |n|< 2^{di}\},\quad i\geq 1. 
  \end{displaymath}
  Then, for any $\eps>0$,
  \begin{equation}
    \label{Bor-Cant_ser}
    \sum_{i=1}^\infty \Prob{\sup_{k\in A(i)}|S_k|/|k|^{1/\beta}>\eps}
    \leq \sum_{i=1}^\infty \Prob{\sup_{k\in A(i)} |S_k|>\eps 
      2^{d(i-1)/\beta}}.
  \end{equation}
  By the multidimensional L\'evy's inequality (\cite[Th.~1]{PP} or
  \cite[Cor.~2.4]{KlesBook}), assuming that $\xi$ is symmetric, we
  have
  \begin{equation}
    \label{Levy}
    \Prob{\sup_{k\in A(i)} |S_k|>\eps 2^{d(i-1)/\beta}}
    \leq 2^d\Prob{|Y_{l_i}|>\eps 2^{d(i-1)/\beta}},
  \end{equation}
  where $Y_{l_i}$ is the sum of $l_i=|m_\cone|2^{di}$ i.i.d.~copies of $\xi$.
  Next, by the one-dimensional L\'evy's inequality,
  \begin{equation}
    \label{change_var}
    \begin{aligned}
      \Prob{|Y_{l_i}|>\eps 2^{d(i-1)/\beta}}
      &=\frac 1{l_{i+1}-l_i}\sum_{j=l_i+1}^{l_{i+1}}
      \Prob{|Y_{l_i}|>\eps 2^{d(i-1)/\beta}}\\
      &\le\frac 2{|m_\cone|2^{di}(2^d-1)}\sum_{j=l_i+1}^{l_{i+1}}
      \Prob{|Y_j|>\eps 2^{d(i-1)/\beta}}\\
      &\le\frac {2^{d+1}}{2^d-1}\sum_{j=l_i+1}^{l_{i+1}}j^{-1}
      \Prob{|Y_j|>\eps_1j^{1/\beta}}
    \end{aligned}
  \end{equation}
  with $\eps_1=4^{-d/\beta}|m_\cone|^{-1/\beta}\eps$.
  Putting all the above inequalities together and noting that
  \begin{displaymath}
  \sum_{j=1}^{\infty}j^{-1}\Prob{|Y_j|>\eps_1j^{1/\beta}}<\infty
  \end{displaymath}
  by \cite[Th.~1]{BK}, we obtain that the series on the left-hand side
  of \eqref{Bor-Cant_ser} converges for all $\eps>0$, and so the
  Borel-Cantelli lemma applies. The desymmetrisation argument is
  standard (see, e.g., the proof of \cite[Th.~3.2]{Gut}) and completes
  the proof of \eqref{eq:1}.

  Let us now turn to the proof of \eqref{mysectLIL}.
  The proof is divided into two steps. First we show that
  \begin{equation}
    \label{BLIL}
    \limsup_{\cone\ni n\to\infty}
    \frac{|S_\n|}
    {\chi(|n|)}\le C \quad \text{a.s.}
  \end{equation}
  for some $C>0$. 
  Repeating the calculations from
  \eqref{Bor-Cant_ser}--\eqref{change_var} with $\chi(|k|)$ instead of
  $|k|^{1/\beta}$ and $C$ instead of $\eps$, we arrive at the
  inequality
  \begin{displaymath}
    \sum_{i=1}^\infty \Prob{\sup_{k\in A(i)}|S_k|/\chi(|k|)>C}
    \leq \frac {2^{2d+1}}{2^d-1}\sum_{j=l_1+1}^{\infty}j^{-1}\Prob{|Y_j|>C\chi(4^{-d}
      |m_\cone|^{-1}j)}.
  \end{displaymath}
  It follows from \cite[Th.~4]{dav68} that the series on the
  right-hand side converges for all $C>2^d|m_\cone|^{1/2}$. An
  application of the Borel-Cantelli lemma and the desymmetrisation
  argument complete the proof of \eqref{BLIL}.
  
  Next, we prove that
  \begin{equation}
    \label{new_LIL}
    \limsup_{\cone\ni\n\to\infty}\frac{|S_n|}{\chi(|n|)}=1 \quad \text{a.s.}
  \end{equation}
  Fix a $\delta>0$ and consider a further closed convex cone
  $\econe\supset\cone$ such that $\econe\setminus\{0\}\subset\Rp$ and
  \begin{equation}
    \label{lcone}
    (1-\delta')|n|>R_n(\econe)>(1-\delta)|n|\quad\text{for all }\n\in\cone
  \end{equation}
  with some $\delta'\in(0,\delta)$.
  Let $\cecone=\Rp\setminus\econe$. Then
  \begin{equation}
    \label{cones_repr}
    \frac{S_n}{\chi(|n|)}
    =\left(\frac{S_n(\econe)}{\chi(R_n(\econe))}
      +\frac{S_n(\cecone)}{\chi(R_n(\cecone))}\;
      \frac{\chi(R_n(\cecone))}{\chi(R_n(\econe))}\right)
    \frac{\chi(R_n(\econe))}{\chi(|n|)}. 
  \end{equation}
  Note that
  \begin{equation}
    \label{GutSectLil}
    \limsup_{\cone\ni\n\to\infty}
    \frac{|S_n(\econe)|}{\chi(R_n(\econe))}=1
  \end{equation}
  by the sectorial law of the iterated logarithm from
  \cite[Th.~3.1]{GutSect}. Besides, (\ref{lcone}) and (\ref{p}) easily
  imply
  \begin{alignat}{2}
    \label{bounds_1}
    \frac{\chi(R_n(\econe))}{\chi(|n|)}&>\sqrt{1-\delta},&\qquad
    \frac{\chi(R_n(\cecone))}{\chi(R_n(\econe))}&<
    \frac{\sqrt{\delta}}{\sqrt{1-\delta}},\\
    \label{bounds_2}
    \frac{\chi(|\n|)}{\chi(R_n(\cecone))}&<\frac 1{\sqrt{\delta'}},&
    \frac{\chi(R_n(\econe))}{\chi(R_n(\cecone))}&<
    \frac{\sqrt{1-\delta'}}{\sqrt{\delta'}},
  \end{alignat}
  for all $\n\in\cone$ with sufficiently large $|\n|$.  Finally,
  \begin{equation}
    \label{last_bound}
    \frac{|S_n(\cecone)|}{\chi(R_n(\cecone))}\le
    \frac{|S_n|}{\chi(|\n|)}\;\frac{\chi(|\n|)}{\chi(R_n(\cecone))}+
    \frac{|S_n(\econe)|}
    {\chi(R_n(\econe))}\;\frac{\chi(R_n(\econe))}{\chi(R_n(\cecone))}.
  \end{equation}
  As shown above,
  \begin{displaymath}
    \limsup_{\cone\ni\n\to\infty}\frac{|S_n|}{\chi(|\n|)}<\infty.
  \end{displaymath}
  So, (\ref{last_bound}), (\ref{GutSectLil}), and (\ref{bounds_2}) lead to
  \begin{displaymath}
    \limsup_{\cone\ni\n\to\infty}
    \frac{|S_n(\cecone)|}{\chi(R_n(\cecone))}<\infty.
  \end{displaymath}
  Due to (\ref{cones_repr}), the latter along with (\ref{GutSectLil})
  and (\ref{bounds_1}) implies \eqref{new_LIL} since $\delta$ can be
  chosen arbitrarily small.
\end{proof}

\section*{Acknowledgements}

This work was supported by the Swiss National Science Foundation in
the framework of the SCOPES programme, Grant No. IZ73Z0\_152292. The
authors are grateful to Oleg Klesov for bringing to them the
problematic related to multiple sums.


\end{document}